\def\benm{\begin{enumerate}}
\def\eenm{\end{enumerate}}

\def\bal{\begin{align}}
\def\eal{\end{align}}

\documentclass{amsart}
\usepackage{amsmath, amsthm, amscd, amsfonts,amssymb, graphicx}

\newtheorem{theorem}{Theorem}[section]
\newtheorem{lemma}[theorem]{Lemma}
\theoremstyle{definition}

\newtheorem{proposition}[theorem]{Proposition}
\newtheorem{corollary}[theorem]{Corollary}

\theoremstyle{remark}
\newtheorem{remark}[theorem]{Remark}

\numberwithin{equation}{section}

\begin{document}

\title[Convolution and Involution on function spaces of homogeneous spaces]{Convolution and Involution on function spaces of homogeneous spaces}

%    Information for first author
\author[Arash Ghaani Farashahi]{Arash Ghaani Farashahi}
%    Address of record for the research reported here
\address{Department of Pure Mathematics, Faculty of Mathematical sciences, Ferdowsi University of
Mashhad (FUM), P. O. Box 1159, Mashhad 91775, Iran.}
\email{ghaanifarashahi.arash@stu-mail.um.ac.ir}
\email{ghaanifarashahi@hotmail.com}
\email{ghaanifarashahi@gmail.com}

%    Current address
\curraddr{}
%    \thanks will become a 1st page footnote.


\date{}

\thanks{E-mail addresses: ghaanifarashahi.arash@stu-mail.um.ac.ir, ghaanifarashahi@hotmail.com (Arash Ghaani Farashahi)}

\begin{abstract}
Let $G$ be a locally compact group and also let $H$ be a compact subgroup of $G$. It is shown that, if $\mu$ is a relatively invariant measure on $G/H$ then there is a well-defined convolution on $L^1(G/H,\mu)$ such that the Banach space $L^1(G/H,\mu)$ becomes a Banach algebra.
We also find a generalized definition of this convolution for other $L^p$-spaces. Finally, we show that various types of involutions can be considered on $G/H$.
\end{abstract}

\maketitle

\section{\bf{Introduction}}

The theory of convolution on function spaces related to locally compact groups has been studied completely in many basic references of harmonic analysis such as \cite{FollH}, \cite{HR1} or \cite {50}. More precisely, convolution theory of functions defined on locally compact groups is the restriction of the convolution theory of measure algebra related to each locally compact group into the function algebra $L^1(G)$, see \cite{FollH} or \cite{HR1}. That is for each $f,g\in L^1(G)$ the convolution $f\ast g$ is defined for a.e. $x\in G$ via
$$f\ast g(x)=\int_Gf(y)g(y^{-1}x)dy.$$
Another approach to the convolution theory on function spaces related to locally compact group can be found in \cite{50}, which first defines convolution on $\mathcal{C}_c(G)$ the space of all continuous functions with compact support and then by continuity and density of $\mathcal{C}_c(G)$ in $L^1(G)$ extend it to $L^1(G)$. Although these two approaches to the convolution theory are equivalent but it seems that
studying properties of the convolution theory on $\mathcal{C}_c(G)$ is more efficient.

Through the world of harmonic analysis and after locally compact groups, we have objects like $G$-spaces and in special case
transitive $G$-spaces which are well known as homogeneous spaces. Proposition 2.44 of \cite{FollH} guarantee that most of transitive $G$-spaces can be considered as a quotient space $G/H$ for some closed subgroup $H$ of $G$. Although $G/H$ is not group when $H$ is not
normal, but principal part of the classical harmonic analysis on $G$ carries over homogeneous spaces.
Theory of classical harmonic analysis on coset space $G/H$ is quite well studied by several authors (see \cite{FollH}, \cite{Forest}, \cite{50}). In many theories of connecting classical harmonic analysis and also mathematical physics homogeneous spaces placed. Common spaces in physics are locally compact homogeneous spaces of the form $G/H$ where $G$ is a locally compact group and $H$ a compact subgroup of $G$ such as the hypersphere $\mathbb{S}^{n-1}$ which is the homogeneous space ${\rm SO}(n)/{\rm SO}(n-1)$. In view of time-frequency analysis, an appropriate convolution and involution on homogeneous spaces will be applicable for (digital) signal processing or filter design (filtering) of 3D image reconstruction on regular curves and surfaces (see \cite{Dey}).

This article contains 4 sections. Section 2 is devoted to fix notations and also a brief summary on homogeneous spaces. In section 3 using the surjective linear map $T_H:\mathcal{C}_c(G)\to \mathcal{C}_c(G/H)$ we define a well-defined convolution on $\mathcal{C}_c(G/H)$ and due to the continuity of this convolution we extend it to $L^1(G/H,\mu)$, where $\mu$ is a relatively invariant measure on $G/H$. It is also proved that $L^1(G/H,\mu)$ with respect to this convolution becomes a Banach algebra.
We also show that if $\mu$ is a $G$-invariant measure on $G/H$, existence of a well-defined involution on the Banach algebra $L^1(G/H,\mu)$ such that $T_H:L^1(G)\to L^1(G/H,\mu)$ becomes a continuous $*$-homomorphism is a necessary and sufficient condition for the subgroup $H$ to be normal in $G$. Finally, in section 4 we introduce some novel approaches for the concept of involution on $L^1(G/H,\mu)$.

\section{\bf{Preliminaries and notations}}

Let $\mathcal{A}$ be a Banach algebra. A subset $\{e_\alpha\}_{\alpha\in I}$ is said to be a right (resp. left) approximate identity for $\mathcal{A}$ if and only if for all $x\in\mathcal{A}$ satisfies $\lim_{\alpha\in I}\|xe_\alpha-x\|_{\mathcal{A}}=0$ (resp. $\lim_{\alpha\in I}\|e_\alpha x-x\|_{\mathcal{A}}=0$) and also by a two-sided approximate identity we mean a left and also right approximate identity.
If $\mathcal{A}$ is a Banach $*$-algebra, a two-sided approximate identity $\{e_\alpha\}$ is said to be invariant under involution if and only if for all $\alpha\in I$ we have $e_\alpha=e_\alpha^*$.

Let $G$ be a locally compact group with the left Haar measure $dx$ and
$H$ be a closed subgroup of $G$ with the left Haar measure $dh$ also let
$\Delta_G$ (resp. $\Delta_H$) be modular function of $G$ (resp. $H$). The left coset space $G/H$ is considered as a homogeneous space that $G$ acts on it from the left and also $\pi:G\to G/H$ is the surjective canonical mapping defined by $\pi(x)=xH$ for all $x\in G$.
It has been shown that  $\mathcal{C}_c(G/H)$ the space of all complex valued continuous functions on $G/H$ with compact support, consists of all $P_H(f)$ functions,  where $f\in\mathcal{C}_c(G)$ and
\begin{equation}
P_H(f)(xH)=\int_Hf(xh)dh.
\end{equation}
In fact, the mapping $P_H:\mathcal{C}_c(G)\to\mathcal{C}_c(G/H)$ is a surjective bounded linear operator (Proposition 2.48 of \cite{FollH}).

If $\mu$ is a Radon measure on $G/H$ and also $x\in G$, the translation $\mu_x$ of $\mu$ is defined by $\mu_x(E)=\mu(xE)$ for all Borel subset $E$ of $G/H$. The measure $\mu$ on $G/H$ is called $G$-invariant if $\mu_x=\mu$ for all $x\in G$, and also $\mu$ is said to be strongly quasi-invariant, if some continuous function $\theta:G\times G/H\to(0,\infty)$ satisfies
$$d\mu_x(yH)=\theta(x,yH)d\mu(yH)\hspace{0.35cm} {\rm \ for \ all} \ x,y\in G.$$
When the function $\theta(x,.)$ reduce to constants, $\mu$ is called relatively invariant under $G$.
A rho-function for the pair $(G,H)$, is a continuous function $\rho:G\to(0,\infty)$ which satisfies
$\rho(xh)={\Delta_H(h)}{\Delta_G(h)}^{-1}\rho(x),$
for each $x\in G$ and $h\in H$. It has been prove that when $H$ is a closed subgroup of $G$, the pair $(G,H)$ admits a rho-function and also for each rho-function $\rho$ on $G$, there is a strongly quasi-invariant measure $\mu$ on $G/H$ such that for each $f\in\mathcal{C}_c(G)$
$$\int_{G/H}P_H(f)(xH)d\mu(xH)=\int_Gf(x)\rho(x)dx,$$
and also satisfies
$$d\mu_x(yH)=\frac{\rho(xy)}{\rho(y)}d\mu(yH)\hspace{0.35cm} {\rm \ for \ all} \ x,y\in G.$$
If $\mu$ is a relatively invariant measure on $G/H$ arises from a rho-function $\rho$, then for all $x,y\in G$ we have (see \cite{Kam3})
\begin{equation}\label{hp}
\rho(xy)=\frac{\rho(x)\rho(y)}{\rho(e)}.
\end{equation}
Moreover, all strongly quasi invariant measures in $G/H$ arise from rho-functions in this manner and all these measures are strongly equivalent (see \cite{FollH}).
As in \cite{FollH} proved, the homogeneous space $G/H$ has a $G$-invariant measure if and only if the constant function $\rho=1$ is a rho-function for the pair $(G,H)$, or equivalently $\Delta_G|_{H}=\Delta_H$.

If $\mu$ is a strongly quasi invariant measure on $G/H$ which is associate with the rho-function $\rho$ for the pair $(G,H)$, then the mapping $T_H:L^1(G)\to L^1(G/H,\mu)$ defined almost everywhere by
\begin{equation}
T_H(f)(xH)=\int_{H}\frac{f(xh)}{\rho(xh)}dh,
\end{equation}
is a surjective bounded linear operator with $\|T_H\|\le1$ (see \cite{50}) and also satisfies the generalized Mackey-Bruhat formula,
\begin{equation}\label{Weil}
\int_{G/H}T_H(f)(xH)d\mu(xH)=\int_Gf(x)dx,
\end{equation}
which is also well known as the Weil's formula.

The natural action of $G$ on the left coset space $G/H$ induces the left translation for measurable functions on $G/H$. More precisely, the left translation $L_x\varphi$ of a measurable function $\varphi$ on $G/H$
is defined via $L_x\varphi(yH)=\varphi(x^{-1}yH)$ for a.e. $yH\in G/H$ and all $x\in G$.
It can be checked that, if $\varphi$ belongs to $\mathcal{C}_c(G/H)$ then we have $L_x\varphi\in\mathcal{C}_c(G/H)$ and also when $\mu$ is a $G$-invariant measure which arises from the rho-function $\rho=1$, the linear map $T_H$ commutes with left translations.
For all $\varphi\in L^1(G/H,\mu)$, the mapping from $G$ into $L^1(G/H,\mu)$ given by $x\mapsto L_x\varphi$ is continuous and also for all $x\in G$ we have
\begin{equation}
\|L_x\varphi\|_{L^1(G/H,\mu)}=\|\varphi\|_{L^1(G/H,\mu)}.
\end{equation}

If $H$ is a normal subgroup of $G$, the left (right) coset space $G/H$ is a locally compact group and so it possesses a left Haar measure which is clearly $G$-invariant and so that we can assume that the Haar measure arises from the rho-function $\rho=1$.
Due to Theorem 3.5.4 of \cite{50}, the linear operator $T_H$ is a continuous $*$-homomorphism and also
\begin{equation}
\mathcal{J}^1(G,H):=\{f\in L^1(G):T_H(f)=0\},
\end{equation}
is a closed two-sided ideal of $L^1(G)$, for more on this topic see \cite{50}.

We recall that, the Banach space $L^1(G)$ is a Banach $*$-algebra with respect to the convolution and involution given by
\begin{equation}
f\ast g(x)=\int_Gf(y)g(y^{-1}x)dy\hspace{1cm} f^*(x)=\Delta_G(x^{-1})\overline{f(x^{-1})},\hspace{1cm} \forall f,g\in L^1(G).
\end{equation}

\section{\bf{Convolution and involution on homogeneous spaces}}

Throughout this article, we assume that $H$ is a compact subgroup of a locally compact group $G$ with a normalized Haar measure.
In this case, $\Delta_G|_H=\Delta_H=1$ and also each rho-function $\rho$ for the pair $(G,H)$ satisfies $\rho(xh)=\rho(x)$ for all $x\in G$ and $h\in H$. These facts guarantee the existence
of a relatively invariant and also a $G$-invariant measure on the left coset space $G/H$.

In \cite{Kam1} R. Kamyabi-Gol and N. Tavallaei studied the possible convolution and also involution on homogeneous spaces of the form $G/H$ where $H$ is a compact subgroup of a locally compact group $G$. The main technique which they have used, is that to generalize the concepts of convolution and also involution on $L^1(G/H)$ such that the linear map $T_H:L^1(G)\to L^1(G/H)$ be a $*$-homomorphism. We recall that, when $H$ is a closed normal subgroup of a locally compact group $G$ due to Theorem 3.5.4 of \cite{50} the linear map $T_H$ is a bounded $*$-homomorphism.

Due to \cite{Kam1} the convolution (resp. involution) defined for all
$\varphi,\psi\in\mathcal{C}_c(G/H)$ via
$$\varphi\ast\psi=T_H(f\ast g)\ \  (\mathrm{resp.}\ \varphi^*=T_H(f^*)),$$
where $f,g\in \mathcal{C}_c(G)$ with $T_H(f)=\varphi$ and $T_H(g)=\psi$. To show that the convolution is well-defined,
it should be proved that for a fixed $f\in L^1(G)$ if $T_H(f)=0$ then $T_H(f\ast g)=0$ for all $g\in L^1(G)$ and also to check that the involution is well-defined, it should be deduced that $T_H(f)=0$ implies $T_H(f^*)=0$. Because of these challenges, authors in \cite{Kam1} introduce the set
\begin{equation}
P(G/H)=\{\varphi\in\mathcal{C}_c(G/H):\exists \eta\in\mathcal{C}(G/H)\ s.t\ \varphi(x^{-1}H)=\eta(xH)\overline{\varphi(xH)}\ \forall x\in G\}.
\end{equation}
They also claimed that when $\mu$ is a relatively invariant measure on $G/H$, the linear span of $P(G/H)$ is $\|.\|_{L^p(G/H,\mu)}$-dense in $L^p(G/H,\mu)$ (see Proposition 3.3 of \cite{Kam1}). For all $p\ge 1$ the set $P(G/H)$ is contained in the pure subspace
\begin{equation}
A^p(G/H,\mu):=\{\varphi\in L^p(G/H,\mu): L_h\varphi=\varphi\ \forall h\in H\},
\end{equation}
of $L^p(G/H,\mu)$, which is $\|.\|_{L^p(G/H,\mu)}$-closed. Because let $\varphi\in P(G/H)$ be arbitrary and also $\eta\in \mathcal{C}(G/H)$ such that $\varphi(x^{-1}H)=\eta(xH)\overline{\varphi(xH)}$ for each $x\in G$. Now, for all $h\in H$ and also $x\in G$ we have
\begin{align*}
\varphi(hxH)&=\varphi((x^{-1}h^{-1})^{-1}H)
\\&=\eta(x^{-1}h^{-1}H)\overline{\varphi(x^{-1}h^{-1}H)}
\\&=\eta(x^{-1}H)\overline{\varphi(x^{-1}H)}=\varphi(xH).
\end{align*}
Thus, $\varphi$ belongs to $A^p(G/H,\mu)$ and hence $P(G/H)\subseteq A^p(G/H,\mu)$ which implies that the linear span $\langle P(G/H)\rangle$ of $P(G/H)$
is contained in $A^p(G/H,\mu)$, for all $p\ge 1$.
Hence, $\langle P(G/H)\rangle$ can not be dense in $L^p(G/H,\mu)$ because it is contained in a closed proper subspace of $L^p(G/H,\mu)$, unless $H$ be a normal subgroup.

More precisely, due to the above descriptions it can be easily checked that the convolution and also the involution defined in \cite{Kam1} are well-defined if and only if $H$ is a normal subgroup of $G$ and clearly in this case the convolution and also the involution coincide with the standard $*$-algebra structure of $L^1(G/H)$ via the group structure of $G/H$.

To fix an appropriate definition of a convolution, we focus our attention to a special sub-algebra of $L^1(G)$. Let
\begin{equation}
\mathcal{C}_c(G:H):=\{f\in\mathcal{C}_c(G):f(xh)=f(x)\ \forall x\in G\ \forall h\in H\}.
\end{equation}
If $f\in \mathcal{C}_c(G)$ and $g\in\mathcal{C}_c(G:H)$ then $f\ast g\in\mathcal{C}_c(G:H)$
and therefore $\mathcal{C}_c(G:H)$ is a left ideal and also a sub-algebra of $\mathcal{C}_c(G)$. We denote the $\|.\|_{L^1(G)}$-closure of $\mathcal{C}_c(G:H)$ in $L^1(G)$ by $L^1(G:H)$. It can be easily checked that $L^1(G:H)$ is a $\|.\|_{L^1(G)}$-closed left ideal and also a $\|.\|_{L^1(G)}$-closed sub-algebra of $L^1(G)$ and we have
\begin{equation}L^1(G:H)=\{f\in L^1(G):R_hf=f\ \ \forall h\in H\}.
\end{equation}
Due to Theorem 2.43 of \cite{FollH} for all $f\in L^1(G:H)$ and also $x\in G$ we get $L_xf\in L^1(G:H)$.
%----------------------------------------------------------------------------------------------------------------------

In the following proposition, we show that the restriction of $T_H$ into $\mathcal{C}_c(G:H)$ is injective.

%----------------------------------------------------------------------------------------------------------------------

\begin{proposition}\label{1}
{\it Let $H$ be a compact subgroup of a locally compact group $G$ and also let $\mu$ be a relatively invariant measure on $G/H$ which arises from the rho-function $\rho$. Then the following statements hold.
\begin{enumerate}
\item $T_H$ maps ${\mathcal{C}_c(G:H)}$ onto $\mathcal{C}_c(G/H)$.
\item $\mathcal{C}_c(G:H)=\{\varphi_{\pi}:=\rho.\varphi\circ \pi:\varphi\in\mathcal{C}_c(G/H)\}$.
\item $T_H|_{\mathcal{C}_c(G:H)}$ is injective.
\end{enumerate}}
\end{proposition}
\begin{proof}
\begin{enumerate}
\item It is clear that $T_H(\mathcal{C}_c(G:H))\subseteq \mathcal{C}_c(G/H)$. Let $\varphi\in \mathcal{C}_c(G/H)$ and put $\varphi_{\pi}(x):=\rho(x)\varphi(xH)$ for all $x\in G$. Due to the relatively invariance of $\mu$ and compactness of $H$ we get $\varphi_{\pi}\in\mathcal{C}_c(G:H)$ and $T_H(\varphi_{\pi})=\varphi$. More precisely, for all $x\in G$ and $h\in H$ we have
\begin{align*}
\varphi_\pi(xh)&=\rho(xh)\varphi(xhH)
\\&=\rho(x)\varphi(xH)=\varphi_\pi(x).
\end{align*}
Thus, for all $x\in G$ we get
\begin{align*}
T_H(\varphi_\pi)(xH)&=\int_H\frac{\varphi_\pi(xh)}{\rho(xh)}dh
\\&=\rho(x)\int_H\frac{\varphi(xhH)}{\rho(xh)}dh
\\&=\rho(x)\int_H\frac{\varphi(xH)}{\rho(x)}dh=\varphi(xH).
\end{align*}
\item It is clear that $\{\varphi_{\pi}:\varphi\in\mathcal{C}_c(G/H)\}\subseteq \mathcal{C}_c(G:H)$. Now let $f\in\mathcal{C}_c(G:H)$ be given. Hence, for all $x\in G$ and $h\in H$ we get $f(xh)=f(x)$. Then, $T_H(f)_{\pi}=f$. Because for all $x\in G$ we have
\begin{align*}
T_H(f)_{\pi}(x)&=\rho(x)T_H(f)(xH)
\\&=\rho(x)\int_H\frac{f(xh)}{\rho(xh)}dh
\\&=f(x)\rho(x)\int_H\frac{1}{\rho(xh)}dh
\\&=f(x)\rho(x)\int_H\frac{1}{\rho(x)}dh=f(x).
\end{align*}
\item Let $f\in\mathcal{C}_c(G:H)$ and $T_H(f)=0$. Therefore, $T_H(f)_\pi=0$ and also invoking (2) we get
$$f=T_H(f)_\pi=0.$$
\end{enumerate}
\end{proof}

%---------------------------------------------------------------------------------------------------------------------------------

Now we are in the position to define an appropriate well-defined convolution on $\mathcal{C}_c(G/H)$ using the linear map $T_H$.
Let $*:\mathcal{C}_c(G/H)\times\mathcal{C}_c(G/H)\to\mathcal{C}_c(G/H)$ be given by
\begin{equation}\label{1.1}
(\varphi,\psi)\mapsto\varphi\ast\psi:=T_H(\varphi_{\pi}\ast\psi_{\pi}),
\end{equation}
for all $\varphi,\psi\in\mathcal{C}_c(G/H)$ where $\varphi_{\pi}\ast\psi_{\pi}$ is the standard convolution of functions $\varphi_{\pi},\psi_{\pi}$ in $L^1(G)$.

It can be easily checked that (\ref{1.1}) defines a well-defined bilinear map. Using Proposition \ref{1}, the convolution defined in
(\ref{1.1}) satisfies
\begin{equation}\label{1.1.1}
[\varphi\ast\psi]_{\pi}=\varphi_{\pi}\ast\psi_{\pi},
\end{equation}
for all $\varphi,\psi\in \mathcal{C}_c(G/H)$. Because, $[\varphi\ast\psi]_{\pi}$ and $\varphi_{\pi}\ast\psi_{\pi}$ belong to $\mathcal{C}_c(G:H)$ and also we have
\begin{equation}
T_H([\varphi\ast\psi]_{\pi})=T_H(\varphi_{\pi}\ast\psi_{\pi}),
\end{equation}
which implies (\ref{1.1.1}).

%---------------------------------------------------------------------------------------------------------------------------------

Next proposition states a worthwhile property of the convolution defined in (\ref{1.1}).

%---------------------------------------------------------------------------------------------------------------------------------

\begin{proposition}\label{1.4}
{\it Let $H$ be a compact subgroup of a locally compact group $G$ and also let $\mu$ be a relatively invariant measure on $G/H$ which arises from a rho-function $\rho$. Then, for all $\varphi,\psi\in \mathcal{C}_c(G/H)$ the convolution defined in (\ref{1.1}) satisfies
\begin{equation}\label{1.5}
\varphi\ast\psi=T_H(\varphi_{\pi}\ast g),
\end{equation}
for all $g\in \mathcal{C}_c(G)$ with $T_H(g)=\psi$.}
\end{proposition}
\begin{proof}
Let $\varphi,\psi\in \mathcal{C}_c(G/H)$ and also $g\in\mathcal{C}_c(G)$ with $T_H(g)=\psi$. Then, for all $x\in G$ we have
\begin{align*}
T_H(\varphi_{\pi}\ast g)(xH)&=\int_H\frac{\varphi_{\pi}\ast g(xh)}{\rho(xh)}dh
\\&=\int_H\int_G\varphi_{\pi}(y)g(y^{-1}xh)\frac{\rho(e)}{\rho(y)\rho(y^{-1}xh)}dhdy
\\&=\int_G\frac{\varphi_{\pi}(y)\rho(e)}{\rho(y)}\left(\int_H\frac{g(y^{-1}xh)}{\rho(y^{-1}xh)}dh\right)dy
\\&=\int_G\frac{\varphi_{\pi}(y)\rho(e)}{\rho(y)}T_H(g)(y^{-1}xH)dy
\\&=\int_G\frac{\varphi_{\pi}(y)\rho(e)}{\rho(y)}\psi(y^{-1}xH)dy
\\&={\rho(e)^{-1}}\int_G{\varphi_{\pi}(y)}\rho(y^{-1})\psi(y^{-1}xH)dy
\\&={\rho(x^{-1})}{\rho(e)^{-2}}\int_G{\varphi_{\pi}(y)}\rho(y^{-1}x)\psi(y^{-1}xH)dy
\\&={\rho(x^{-1})}{\rho(e)^{-2}}\int_G{\varphi_{\pi}(y)}\psi_{\pi}(y^{-1}x)dy=\rho(x^{-1}){\rho(e)^{-2}}\varphi_{\pi}\ast\psi_{\pi}(x).
\end{align*}
Now for all $x\in G$ we achieve
\begin{align*}
[T_H(\varphi_{\pi}\ast g)]_{\pi}(x)&=\rho(x)T_H(\varphi_{\pi}\ast g)(xH)
\\&=\rho(x)\rho(x^{-1})\rho(e)^{-2}\varphi_{\pi}\ast\psi_{\pi}(x)=\varphi_{\pi}\ast\psi_{\pi}(x),
\end{align*}
which clearly implies (\ref{1.5}).
\end{proof}

%--------------------------------------------------------------------------------------------------------------------------------------

In the following theorem, we will show that the convolution defined in (\ref{1.1}) can be extended to a convolution on $L^1(G/H,\mu)$ such that  $L^1(G/H,\mu)$ becomes a Banach algebra.

%---------------------------------------------------------------------------------------------------------------------------------

\begin{theorem}\label{2}
Let $H$ be a compact subgroup of a locally compact group $G$ and also let $\mu$ be a relatively invariant measure on $G/H$ which arises from a rho-function $\rho$.
The convolution defined in (\ref{1.1}) can be uniquely extended to a convolution
$$\ast:L^1(G/H,\mu)\times L^1(G/H,\mu)\to L^1(G/H,\mu),$$
such that $L^1(G/H,\mu)$ becomes a Banach algebra.
\end{theorem}
\begin{proof}
Associativity of the standard convolution on $L^1(G)$ guarantees associativity of (\ref{1.1}). Let $\varphi\in\mathcal{C}_c(G/H)$ be arbitrary. Using the Weil's formula (\ref{Weil}) we have
\begin{align*}
\|\varphi\|_{L^1(G/H,\mu)}&=\int_{G/H}|\varphi(xH)|d\mu(xH)
\\&=\int_{G/H}T_H(|\varphi_{\pi}|)(xH)d\mu(xH)
\\&=\int_G|\varphi_{\pi}(x)|dx=\|\varphi_{\pi}\|_{L^1(G)}.
\end{align*}
Now let $\varphi,\psi\in \mathcal{C}_c(G/H)$. Due to (\ref{1.1}) and also preceding calculations we achieve
\begin{align*}
\|\varphi\ast\psi\|_{L^1(G/H,\mu)}&=\|T_H(\varphi_{\pi}\ast\psi_{\pi})\|_{L^1(G/H,\mu)}
\\&\le\|\varphi_{\pi}\ast\psi_{\pi}\|_{L^1(G)}
\\&\le\|\varphi_{\pi}\|_{L^1(G)}\|\psi_{\pi}\|_{L^1(G)}
=\|\varphi\|_{L^1(G/H,\mu)}\|\psi_{\pi}\|_{L^1(G/H,\mu)}.
\end{align*}
Thus, using continuity we can uniquely extend the convolution $\ast:\mathcal{C}_c(G/H)\times\mathcal{C}_c(G/H)\to\mathcal{C}_c(G/H)$ to a convolution $\ast:L^1(G/H,\mu)\times L^1(G/H,\mu)\to L^1(G/h,\mu)$ which still for all $\varphi,\psi\in L^1(G/H,\mu)$ satisfies
\begin{equation}\label{}
\|\varphi\ast\psi\|_{L^1(G/H,\mu)}\le \|\varphi\|_{L^1(G/H,\mu)}\|\psi\|_{L^1(G/H,\mu)}.
\end{equation}
\end{proof}

%-------------------------------------------------------------------------------------------------------------------------------------

The following corollary is an immediate consequence of Theorem \ref{2}.

\begin{corollary}
{\it Let $H$ be a compact subgroup of a locally compact group $G$ also let $\mu$ be a relatively invariant measure on $G/H$ which arises from a rho-function.
Then, the linear map
$$T_H:L^1(G:H)\to L^1(G/H,\mu),$$
is an isometric isomorphism.
}\end{corollary}

%--------------------------------------------------------------------------------------------------------------------------------

We can also deduce the following property of the convolution and also left translations on $L^1(G/H,\mu)$ for a $G$-invariant measure $\mu$ on $G/H$.

%--------------------------------------------------------------------------------------------------------------------------------

\begin{corollary}\label{LT0}
{\it Let $H$ be a compact subgroup of a locally compact group $G$ and also $\mu$ be a $G$-invariant measure on $G/H$. Then, for all $x\in G$ and $\varphi,\psi\in L^1(G/H,\mu)$ we have
\begin{equation}\label{LT}
L_x(\varphi\ast\psi)=(L_x\varphi)\ast\psi.
\end{equation}
}\end{corollary}
\begin{proof}
If $\mu$ is a $G$-invariant measure on $G/H$ we have $T_H=P_H$. Therefore, left translations commute $T_H$. Thus, for all $x\in G$ and also $\varphi,\psi\in L^1(G/H,\mu)$ we have
\begin{align*}
L_x(\varphi\ast\psi)&=L_x\left(T_H(\varphi_{\pi}\ast\psi_{\pi})\right)
\\&=T_H\left(L_x(\varphi_{\pi}\ast\psi_{\pi})\right)
\\&=T_H\left(L_x(\varphi_{\pi})\ast\psi_{\pi}\right)
\\&=T_H\left((L_x\varphi)_{\pi}\ast\psi_{\pi}\right)=(L_x\varphi)\ast\psi.
\end{align*}
\end{proof}

%------------------------------------------------------------------------------------------------------------------------------------

In the following proposition we will show that the Banach algebra mentioned in Theorem \ref{2} always possesses a right approximate identity.

%-------------------------------------------------------------------------------------------------------------------------------------

\begin{proposition}\label{right}
{\it Let $H$ be a compact subgroup of a locally compact group $G$ and also let $\mu$ be a relatively invariant measure on $G/H$. The Banach algebra $L^1(G/H,\mu)$ possesses a right approximate identity.}
\end{proposition}
\begin{proof}
Let $\{k_\alpha\}_{\alpha\in I}$ be an approximate identity for $L^1(G)$ according to the proposition 2.42 of \cite{FollH} and also for all $\alpha\in I$ let $\psi_\alpha:=T_H(k_\alpha)$. Now due to Proposition \ref{1.4} for all $\varphi\in L^1(G/H,\mu)$ we have
\begin{align*}
\lim_{\alpha\in I}\|\varphi\ast\psi_\alpha-\varphi\|_{L^1(G/H,\mu)}
&=\lim_{\alpha\in I}\|T_H(\varphi_{\pi}\ast k_\alpha)-T_H(\varphi_\pi)\|_{L^1(G/H,\mu)}
\\&=\lim_{\alpha\in I}\|T_H(\varphi_{\pi}\ast k_\alpha-\varphi_\pi)\|_{L^1(G/H,\mu)}
\\&\le\lim_{\alpha\in I}\|\varphi_{\pi}\ast k_\alpha-\varphi_{\pi}\|_{L^1(G)}=0.
\end{align*}
\end{proof}

%-------------------------------------------------------------------------------------------------------------------------------------

A natural and also unprofessional approach for finding an involution on $\mathcal{C}_c(G/H)$ may candidates $T_H([\varphi_{\pi}]^*)$ as $\varphi^*$,
where $[\varphi_{\pi}]^*$ is the standard involution of $\varphi_{\pi}$ in $L^1(G)$.
Although this definition is well-define but it can be easily checked that it fails to satisfies basic properties of an involution such as $\varphi ^{**}=\varphi$ or anti-homomorphism property.
These problems and challenges occur because, $\mathcal{C}_c(G:H)$ and also $L^1(G:H)$ are not invariant under the standard involution of $L^1(G)$. We recall that, $\mathcal{C}_c(G:H)$ plays an important role on this theory because $T_H|_{\mathcal{C}_c(G:H)}$ is injective. Thus, to define an appropriate involution on $\mathcal{C}_c(G/H)$ we need a technical approach.

%---------------------------------------------------------------------------------------------------------------------

In the next theorem we show that if $\mu$ is a $G$-invariant measure on $G/H$, existence of a well-defined involution on the Banach algebra $L^1(G/H,\mu)$ such that $T_H:L^1(G)\to L^1(G/H,\mu)$ becomes a continuous $*$-homomorphism is a necessary and sufficient condition for the subgroup $H$ to be normal in $G$. To this, we need the following lemma.

%---------------------------------------------------------------------------------------------------------------------

\begin{lemma}\label{3.F}
Let $H$ be a compact subgroup of a locally compact group $G$ and $\mu$ be a $G$-invariant measure on $G/H$ also let $f\in L^1(G:H)$ with $f-f^*\in\mathcal{J}^1(G,H)$. Then, for all $h\in H$ we have
$L_hf-f\in\mathcal{J}^1(G,H)$.
\end{lemma}
\begin{proof}
Let $f\in L^1(G:H)$ with $f-f^*\in\mathcal{J}^1(G,H)$. Since $f\in L^1(G:H)$ and
also due to compactness of $H$ we achieve $L_hf^*=f^*$ in $L^1(G)$ for all $h\in H$. Because, for all $h\in H$ and also a.e $x\in G$ we have
\begin{align*}
L_hf^*(x)&=f^*(h^{-1}x)
\\&=\Delta_G(x^{-1}h)\overline{f(x^{-1}h)}
\\&=\Delta_G(x^{-1})\overline{f(x^{-1})}
=f^*(x).
\end{align*}
Now since $T_H$ commutes with left translations, for all $h\in H$ we get
\begin{align*}
T_H(L_hf)&=L_hT_H(f)
\\&=L_hT_H(f^*)
\\&=T_H(L_hf^*)
\\&=T_H(f^*)=T_H(f).
\end{align*}
\end{proof}

%----------------------------------------------------------------------------------------------------------------------

\begin{theorem}\label{3}
Let $H$ be a compact subgroup of a locally compact group $G$ and also let $\mu$ be a $G$-invariant measure on $G/H$.
There is a well-defined involution on the Banach algebra $L^1(G/H,\mu)$ such that the linear map
$$T_H:L^1(G)\to L^1(G/H,\mu),$$
becomes a surjective continuous $*$-homomorphism if and only if $H$ is a normal subgroup of $G$.
\end{theorem}
\begin{proof}
Assume that there is a well-defined involution on the Banach algebra $L^1(G/H,\mu)$ such that the linear map $T_H:L^1(G)\to L^1(G/H,\mu)$ becomes a continuous $*$-homomorphism and also let $\{k_\alpha\}_{\alpha\in I}$ be an invariant under involution two-sided approximate identity for $L^1(G)$. Now, for all $\alpha\in I$ let $\psi_\alpha:=T_H(k_\alpha)$. Proposition \ref{right} implies that $\{\psi_\alpha\}_{\alpha\in I}$ is a right approximate identity. It is also a left approximate identity, because for all $\varphi\in L^1(G/H,\mu)$ and also $f\in L^1(G)$ with $T_H(f)=\varphi$ we have
\begin{align*}
\lim_{\alpha\in I}\|\psi_\alpha\ast\varphi-\varphi\|_{L^1(G/H,\mu)}
&=\lim_{\alpha\in I}\|T_H(k_\alpha)\ast T_H(f)-T_H(f)\|_{L^1(G/H,\mu)}
\\&=\lim_{\alpha\in I}\|T_H(k_\alpha\ast f-f)\|_{L^1(G/H,\mu)}
\\&\le\lim_{\alpha\in I}\|k_\alpha\ast f-f\|_{L^1(G)}=0.
\end{align*}
Note that, since $T_H$ is a $*$-homomorphism $\{\psi_\alpha\}_{\alpha\in I}$ is an invariant under involution approximate identity and also for $\alpha\in I$
we have $[\psi_\alpha]_{\pi}=\psi_\alpha\circ\pi$. Thus, for all $\alpha\in I$ we get
$[\psi_\alpha]_{\pi}\in L^1(G:H)$ and since $\psi_\alpha=\psi_\alpha^*$ we also have $[\psi_\alpha]_{\pi}-[\psi_\alpha]_{\pi}^*\in\mathcal{J}^1(G,H)$. Because,
\begin{align*}
T_H([\psi_\alpha]_{\pi}^*)&=T_H([\psi_\alpha]_{\pi})^*\\&=\psi_\alpha^*
\\&=\psi_\alpha=T_H([\psi_\alpha]_{\pi}).
\end{align*}
Therefore, for all $\alpha\in I$, Lemma \ref{3.F} works and for all $h\in H$ we have $L_h[\psi_\alpha]_{\pi}-[\psi_\alpha]_{\pi}\in \mathcal{J}^1(G,H)$
which implies that $L_h\psi_\alpha=\psi_\alpha$ for all $h\in H$.
Now let $\varphi\in\mathcal{C}_c(G/H)$ be arbitrary and also $h\in H$. Using preceding calculations, continuity of $L_h$, (\ref{LT})
and also since $T_H$ commutes $L_h$
we get
\begin{align*}
L_h\varphi
&=L_h\left(\lim_{\alpha\in I}\psi_\alpha\ast\varphi\right)
\\&=\lim_{\alpha\in I}L_h\left(\psi_\alpha\ast\varphi\right)
\\&=\lim_{\alpha\in I}\left(L_h\psi_\alpha\right)\ast\varphi
\\&=\lim_{\alpha\in I}\psi_\alpha\ast\varphi=\varphi.
\end{align*}
Thus, we achieve that $\varphi(hxH)=\varphi(xH)$ for all $x\in G$ and also all $h\in H$. Since $\mathcal{C}_c(G/H)$ separates points of $G/H$
we deduced that $hxH=xH$ for all $h\in H$ and $x\in G$, which implies that $H$ is a normal subgroup of $G$.
If $H$ is a normal subgroup, The fact that $G/H$ is a locally compact group and Theorem 3.5.4 of \cite{50} guarantee existence of a well-defined involution on $L^1(G/H,\mu)$
and also the $*$-homomorphism property of $T_H$.
\end{proof}

%------------------------------------------------------------------------------------------------------------------------------------

In the following corollary we also deduce that if we replace the convolution given in (\ref{1.1}) by any well-defined convolution such that $T_H$ still be a continuous $*$-homomorphism, then $H$ is automatically normal.

%------------------------------------------------------------------------------------------------------------------------------------

\begin{corollary}
{\it Let $H$ be a compact subgroup of a locally compact group $G$ and also let $\mu$ be a $G$-invariant measure on $G/H$.
There is a well-defined convolution and involution on $L^1(G/H,\mu)$ such that the linear map
$$T_H:L^1(G)\to L^1(G/H,\mu),$$
becomes a continuous $*$-homomorphism if and only if $H$ is a normal subgroup of $G$.
}\end{corollary}

%------------------------------------------------------------------------------------------------------------------------------------

In the sequel we find an appropriate definition of the convolution (\ref{1.1}) for other $L^p$-spaces related to homogeneous spaces.
First we need a generalized notation of the linear map $T_H$ for other $L^p$-spaces.

%---------------------------------------------------------------------------------------------------------------------------------------

\begin{proposition}\label{3.1}
{\it Let $H$ be a compact subgroup of a locally compact group $G$, also let $\mu$ be a $G$-invariant measure on $G/H$ and $p\ge 1$. The linear map $T_H:\mathcal{C}_c(G)\to\mathcal{C}_c(G/H)$ has a unique extension to a bounded linear map from $L^p(G)$ onto $L^p(G/H,\mu)$.
}\end{proposition}
\begin{proof}
Let $p\ge 1$ and $f\in\mathcal{C}_c(G)$. Using the compactness of $H$ and also the Weil's  formula (\ref{Weil}) we get
\begin{align*}
\left\|T_H(f)\right\|_{L^p(G/H,\mu)}^p&=\int_{G/H}|T_H(f)(xH)|^pd\mu(xH)
\\&\le\int_{G/H}\left(\int_H{|f(xh)|}dh\right)^pd\mu(xH)
\\&\le\int_{G/H}\int_H{|f(xh)|}^pdhd\mu(xH)=\|f\|_{L^p(G)}^p.
\end{align*}
Now due to continuity of $T_H$ we can uniquely extend it into a bounded linear operator from $L^p(G)$ onto $L^p(G/H,\mu)$ which still satisfies
$\|T_H(f)\|_{L^p(G/H,\mu)}\le\|f\|_{L^p(G)}$ for all $f\in L^p(G)$.
\end{proof}

%---------------------------------------------------------------------------------------------------------------------------------

\begin{corollary}\label{3.2}
{\it Let $H$ be a compact subgroup of a locally compact group $G$, also let $\mu$ be a $G$-invariant measure on $G/H$ and $p\ge 1$. Then, for all $\varphi\in L^p(G/H,\mu)$ we have $\varphi_{\pi}\in L^p(G)$ and also
\begin{equation}
\|\varphi\|_{L^p(G/H,\mu)}=\|\varphi_{\pi}\|_{L^p(G)}.
\end{equation}}
\end{corollary}

%----------------------------------------------------------------------------------------------------------------------------------

\begin{corollary}
{\it Let $H$ be a compact subgroup of a locally compact group $G$, also let $\mu$ be a $G$-invariant measure on $G/H$ and $p\ge 1$. Then, for all $\varphi\in L^p(G/H,\mu)$ we have
\begin{equation}
\lim_{x\to e}\|L_x\varphi-\varphi\|_{L^p(G/H,\mu)}=0.
\end{equation}}
\end{corollary}

%----------------------------------------------------------------------------------------------------------------------------------

Now let $\varphi\in L^1(G/H,\mu)$ and $\psi\in L^p(G/H,\mu)$ with $p\ge 1$. Define
\begin{equation}
\varphi\ast\psi:=T_H(\varphi_{\pi}\ast\psi_{\pi}).
\end{equation}
Then, $L^p(G/H,\mu)$ becomes a left Banach $L^1(G/H,\mu)$-module via the module action
\begin{equation}\label{DD}
*:L^1(G/H,\mu)\times L^p(G/H,\mu)\to L^p(G/H,\mu),
\end{equation}
that is defined by $(\varphi,\psi)\mapsto\varphi\ast\psi$. Using Proposition \ref{3.1}, Corollary \ref{3.2} and also Proposition 2.39 of \cite{FollH}, for all $\varphi\in L^1(G/H,\mu)$ and $\psi\in L^p(G/H,\mu)$ we have
\begin{align*}
\|\varphi\ast\psi\|_{L^p(G/H,\mu)}&=\|T_H(\varphi_{\pi}\ast\psi_{\pi})\|_{L^p(G/H,\mu)}
\\&\le \|\varphi_{\pi}\ast\psi_{\pi}\|_{L^p(G)}
\\&\le \|\varphi_{\pi}\|_{L^1(G)}\|\psi_{\pi}\|_{L^p(G)}=\|\varphi\|_{L^1(G/H,\mu)}\|\psi\|_{L^p(G/H,\mu)}.
\end{align*}

%---------------------------------------------------------------------------------------------------------------------------------------

Thus, we prove the following theorem.

%---------------------------------------------------------------------------------------------------------------------------------------

\begin{theorem}
Let $H$ be a compact subgroup of a locally compact group $G$ and also let $\mu$ be a $G$-invariant measure on $G/H$.
Then, $L^p(G/H,\mu)$ via the module action defined in (\ref{DD}) is a Banach $L^1(G/H,\mu)$-module for all $p\ge 1$.
\end{theorem}

%---------------------------------------------------------------------------------------------------------------------------------------

\section{\bf{Approaches for the involution of $L^1(G/H,\mu)$}}

%---------------------------------------------------------------------------------------------------------------------------------------

In this section we study approaches for the concept of involution on the $L^1$-function space of the homogeneous space $G/H$. Recall that we still assume that $H$ is a compact subgroup of a locally compact group $G$ with a normalized Haar measure.

In the first approach we show that when $\mu$ is a $G$-invariant measure on $G/H$, there is a closed sub-algebra $A^1(G/H,\mu)$ of
$L^1(G/H,\mu)$ and also a well-defined involution on this closed sub-algebra in which with respect to this involution and also the induced convolution from $L^1(G/H,\mu)$, the closed sub-algebra $A^1(G/H,\mu)$ will be a Banach $*$-algebra.

\subsection{Approach I}

For all $p\ge 1$ and also an arbitrary $G$-invariant measure $\mu$ on $G/H$ let
\begin{equation}
A^p(G/H,\mu):=\{\varphi\in L^p(G/H,\mu): L_h\varphi=\varphi\ \ \forall h\in H\}.
\end{equation}
It can be checked that for all $p\ge 1$, $A^p(G/H,\mu)$ is a $\|.\|_{L^p(G/H,\mu)}$-closed subspace of $L^p(G/H,\mu)$. If we also put
\begin{equation}
A^p(G:H):=\{f\in L^p(G): L_hf=f\ \ \forall h\in H\},
\end{equation}
then $A^p(G:H)$ is a $\|.\|_{L^p(G)}$-closed subspace of $L^p(G)$ and also $T_H$ maps $A^p(G:H)$ onto $A^p(G/H,\mu)$.

If $p=1$, then $A^1(G:H)$ is a $\|.\|_{L^1(G)}$-closed right ideal of $L^1(G)$
and also using Corollary \ref{LT0}, $A^1(G/H,\mu)$ is a $\|.\|_{L^1(G/H,\mu)}$-closed right ideal and so that $A^1(G/H,\mu)$ is a
$\|.\|_{L^1(G/H,\mu)}$-closed sub-algebra of $L^1(G/H,\mu)$.
Therefore, $A^1(G/H,\mu)$ is a Banach algebra.  Now we define the involution map $^*:{A^1(G/H,\mu)}\to A^1(G/H,\mu)$ via $\varphi\mapsto\varphi^*$ by
\begin{equation}\label{INV1}
\varphi^*:=T_H([\varphi_{\pi}]^*).
\end{equation}
It is clear that $^*:{A^1(G/H,\mu)}\to A^1(G/H,\mu)$ is a well-defined conjugate linear map. For all $\varphi\in A^1(G/H,\mu)$ we have
$\varphi_{\pi}\in A^1(G:H)\cap L^1(G:H)$ which implies that $[\varphi_{\pi}]^*$ belongs $A^1(G:H)\cap L^1(G:H)$.
Hence, for all $\varphi\in A^1(G/H,\mu)$ we achieve
\begin{equation}\label{qi}
[\varphi^*]_{\pi}=[\varphi_{\pi}]^*.
\end{equation}

%------------------------------------------------------------------------------------------------------------------------------

In the following theorem, it is shown that the Banach algebra  $A^1(G/H,\mu)$ with respect to the involution defined in (\ref{INV1}) is a Banach $*$-algebra.

%------------------------------------------------------------------------------------------------------------------------------

\begin{theorem}
Let $H$ be a compact subgroup of a locally compact group $G$ and also let $\mu$ be a $G$-invariant measure on $G/H$. The Banach algebra $A^1(G/H,\mu)$ with respect to the involution defined in (\ref{INV1}) is a Banach $*$-algebra.
\end{theorem}
\begin{proof}
Let $\varphi\in A^1(G/H,\mu)$ be arbitrary. Using (\ref{qi}), we have
\begin{align*}
[\varphi^{**}]_{\pi}&=[(\varphi^*)^*]_{\pi}
\\&=[(\varphi^*)_{\pi}]^*
\\&=[(\varphi_{\pi})^*]^*=\varphi_{\pi},
\end{align*}
which implies $\varphi^{**}=\varphi$. According to the anti-homomorphism property of the standard involution on the Banach algebra $L^1(G)$, (\ref{1.1.1}) and (\ref{qi}) for all $\varphi,\psi\in A^1(G/H,\mu)$ we achieve
\begin{align*}
[(\varphi\ast\psi)^*]_{\pi}&=[(\varphi\ast\psi)_{\pi}]^*
\\&=[\varphi_{\pi}\ast\psi_{\pi}]^*
\\&=[\psi_{\pi}]^*\ast[\varphi_{\pi}]^*
\\&=[\psi^*]_{\pi}\ast[\varphi^*]_{\pi}=[\psi^*\ast\varphi^*]_{\pi},
\end{align*}
which guarantees that $(\varphi\ast\psi)^*=\psi^*\ast\varphi^*$. The conjugate linear map $^*:A^1(G/H,\mu)\to A^1(G/H,\mu)$ is an isometric, because we have
\begin{align*}
\|\varphi^*\|_{L^1(G/H,\mu)}&=\left\|[\varphi^*]_{\pi}\right\|_{L^1(G)}
\\&=\left\|[\varphi_{\pi}]^*\right\|_{L^1(G)}
\\&=\|\varphi_{\pi}\|_{L^1(G)}=\|\varphi\|_{L^1(G/H,\mu)}.
\end{align*}
\end{proof}

%--------------------------------------------------------------------------------------------------------------------------------

\begin{corollary}
{\it Let $H$ be a compact subgroup of a locally compact group $G$ and also let $\mu$ be a $G$-invariant measure on $G/H$.
Then, $T_H:A^1(G:H)\cap L^1(G:H)\to A^1(G/H,\mu)$ is a continuous $*$-homomorphism.}
\end{corollary}

%--------------------------------------------------------------------------------------------------------------------------------

\begin{remark}
If $H$ is a compact normal subgroup of a locally compact group $G$, automatically for all $p\ge 1$ we have $A^p(G/H,\mu)=L^p(G/H,\mu)$ and also for $p=1$ the involution defined in (\ref{qi}) coincides with the standard involution on $L^1(G/H,\mu)$ and therefore the Banach $*$-algebra $A^1(G/H,\mu)$ coincides with $L^1(G/H,\mu)$.
\end{remark}

%--------------------------------------------------------------------------------------------------------------------------------

\subsection{Approach II}

In this approach we note that for each closed subgroup $H$ of a locally compact group $G$, left coset space $G/_lH=\{xH:x\in G\}$ and
right coset space $G/_rH=\{Hx:x\in G\}$ are topologically the same via the homeomorphism $Q:G/_lH\to G/_rH$ given by $Q(xH)=Hx^{-1}$. The technical point is that when $H$ is a normal closed subgroup of $G$, these space are precisely the same, because each left coset $xH$ is precisely the right coset $Hx$. Thus, due to this approach we should fix differences in notations for left coset space and also right coset space.

From now on by $\mu_l$ and $\mu_r$ we mean a measure on $G/_lH$ respectively $G/_rH$ also $\pi_l:G\to G/_lH$ and $\pi_r:G\to G/_rH$ denote the associated canonical maps from $G$ onto coset spaces $G/_lH$ and $G/_rH$ respectively. The same terminologies as fixed for the left coset space $G/_lH$, similarly can be used for $G/_rH$ such as $L^1_r(G:H)$. It is worthwhile to remember that
\begin{equation}
L^1_l(G:H)=\{f\in L^1(G):R_hf=f\ \forall\ h\in H\}
\end{equation}
\begin{equation}
L^1_r(G:H)=\{f\in L^1(G):L_hf=f\ \forall\ h\in H\}.
\end{equation}
A $G$-invariant measure $\mu_r$ on $G/_rH$ stands for a Radon measure $\mu_r$ on $G/_rH$ satisfying $\mu_r(Ex)=\mu_r(E)$ for all $x\in G$ and Borel subset $E$ of $G/_rH$. If $\mu_r$ is a $G$-invariant measure on $G/_rH$, the linear map
\begin{equation}
T_H^r:L^1(G)\to L^1(G/_rH,\mu_r),
\end{equation}
for all $f\in L^1(G)$ is given by
\begin{equation}
T_H^r(f)(Hx)=\int_Hf(hx)dh.
\end{equation}
We should also recall that the convolution given in (\ref{1.1}) can be similarly defined in the case of right coset space for $L^1$-function space $L^1(G/_rH,\mu_r)$, where $\mu_r$ is a relatively invariant measure on $G/_rH$. We use the notation $\ast_r$ for the convolution satisfying \begin{equation}
\varphi\ast_r\psi=T_H^r(\varphi_{\pi_r}\ast\psi_{\pi_r}),
\end{equation}
for all $\varphi,\psi\in L^1(G/_rH,\mu_r)$.

%------------------------------------------------------------------------------------------------------------------------

Let $\mu_l$ and $\mu_r$ be $G$-invariant measures on $G/_lH$ respectively $G/_rH$. Now let $\varphi\in L^1(G/_lH,\mu_l)$ with $\varphi=T_H^l(f)$
for some $f\in L^1(G)$. We put
\begin{equation}\label{inv}
\varphi^{*_{l,r}}:=T_H^r(f^*),
\end{equation}
where $f^*$ is the standard involution on $L^1(G)$. The conjugate linear (involution-type) map
\begin{equation}\label{inv1}
^{\ast_{l,r}}:L^1(G/_lH,\mu_l)\to L^1(G/_rH,\mu_r),
\end{equation}
given by
$\varphi\mapsto \varphi^{\ast_{l,r}}$ is well defined. Because using compactness of $H$ if $T_H^l(f)=0$ for some $f\in L^1(G)$, then for a.e $Hx\in G/_r H$ we have
\begin{align*}
T_H^r(f^*)(Hx)&=\int_Hf^*(hx)dh
\\&=\int_H\Delta_G(x^{-1}h^{-1})\overline{f(x^{-1}h^{-1})}dh
\\&=\Delta_G(x^{-1})\int_H\overline{f(x^{-1}h)}dh
\\&=\Delta_G(x^{-1})\overline{T_H^l(f)(x^{-1}H)}=0.
\end{align*}

The involution-type map $^{\ast_{l,r}}:L^1(G/_lH,\mu_l)\to L^1(G/_rH,\mu_r)$ for all $\varphi\in L^1(G/_lH,\mu_l)$ satisfies
\begin{equation}\label{qlr1}
[\varphi_{\pi_l}]^*=[\varphi^{\ast_{l,r}}]_{\pi_r}.
\end{equation}
Because, due to (\ref{inv}) we have
$$T_H^r([\varphi_{\pi_l}]^*)=\varphi^{*_{l,r}}=T_H^r([\varphi^{\ast_{l,r}}]_{\pi_r}),$$
and since $[\varphi_{\pi_l}]^*,[\varphi^{\ast_{l,r}}]_{\pi_r}$ belong to $L^1_r(G:H)$ and $T_H^r|_{L^1_r(G:H)}$ is injective, (\ref{qlr1}) holds.

The involution-type map
\begin{equation}\label{inv2}
^{\ast_{r,l}}:L^1(G/_rH,\mu_r)\to L^1(G/_lH,\mu_l),
\end{equation}
can be defined similarly which satisfies $[\phi_{\pi_r}]^*=[\phi^{\ast_{r,l}}]_{\pi_l}$ for all $\phi\in L^1(G/_rH,\mu_r)$.
Using (\ref{qlr1}) for all $\varphi\in L^1(G/_lH,\mu_l)$ and $\phi\in L^1(G/_rH,\mu_r)$ we have $\left(\varphi^{*_{l,r}}\right)^{*_{r,l}}=\varphi$ and also $\left(\phi^{*_{r,l}}\right)^{*_{l,r}}=\phi$. Because
\begin{align*}
\left[\left(\varphi^{*_{l,r}}\right)^{*_{r,l}}\right]_{\pi_l}
&=\left[\left(\varphi^{\ast_{l,r}}\right)_{\pi_r}\right]^*
\\&=\left[\left[\varphi_{\pi_l}\right]^*\right]^*=\varphi_{\pi_l}.
\end{align*}

%---------------------------------------------------------------------------------------------------------------

In the sequel theorem we show that the linear map $^{\ast_{l,r}}:L^1(G/_lH,\mu_l)\to L^1(G/_rH,\mu_r)$ is an anti-homomorphism isometric.

%---------------------------------------------------------------------------------------------------------------

\begin{theorem}
Let $H$ be a compact subgroup of a locally compact group $G$ also let $\mu_l$ and $\mu_r$ be arbitrary $G$-invariant measures on $G/_lH$ and  $G/_rH$ respectively.
The involution-type map
$$^{*_{l,r}}:L^1(G/_lH,\mu_l)\to L^1(G/_rH,\mu_r),$$
is an anti-homomorphism isometric.
\end{theorem}
\begin{proof}
Let $\varphi,\psi\in L^1(G/_lH,\mu_l)$ be arbitrary. Equivalently, it is sufficient to show that
\begin{equation}
[\left(\varphi\ast_l\psi\right)^{*_{l,r}}]_{\pi_r}=[\psi^{*_{l,r}}\ast_r\varphi^{*_{l,r}}]_{\pi_r}.
\end{equation}
Due to the anti-homomorphism property of the standard involution on $L^1(G)$ also (\ref{1.1.1}) and (\ref{qlr1}) we have
\begin{align*}
[(\varphi\ast_l\psi)^{\ast_{l,r}}]_{\pi_r}
&=[(\varphi \ast_l\psi)_{\pi_l}]^*
\\&=[\varphi_{\pi_l}\ast\psi_{\pi_l}]^*
\\&=[\psi_{\pi_l}]^*\ast[\varphi_{\pi_l}]^*
\\&=[\psi^{\ast_{l,r}}]_{\pi_r}\ast[\varphi^{\ast_{l,r}}]_{\pi_r}
=[\psi^{*_{l,r}}\ast_r\varphi^{*_{l,r}}]_{\pi_r}.
\end{align*}
Hence, we achieve
\begin{align*}
(\varphi\ast_l\psi)^{\ast_{l,r}}&=T_H^r([(\varphi\ast_l\psi)^{\ast_{l,r}}]_{\pi_r})
\\&=T_H^r([\psi^{*_{l,r}}\ast_r\varphi^{*_{l,r}}]_{\pi_r})=\psi^{*_{l,r}}\ast_r\varphi^{*_{l,r}}.
\end{align*}
For all $\varphi\in L^1(G/_lH,\mu_l)$ we have
\begin{align*}
\|\varphi^{\ast_{l.r}}\|_{L^1(G/_rH,\mu_r)}&=\|[\varphi^{*_{l,r}}]_{\pi_r}\|_{L^1(G)}
\\&=\|[\varphi_{\pi_l}]^*\|_{L^1(G)}
\\&=\|\varphi_{\pi_l}\|_{L^1(G)}
\\&=\|T_H^l(\varphi_{\pi_l})\|_{L^1(G/_lH,\mu_l)}=\|\varphi\|_{L^1(G/_lH,\mu_l)}
\end{align*}
\end{proof}

%---------------------------------------------------------------------------------------------------------------------------------------

\begin{corollary}
{\it Let $H$ be a compact subgroup of a locally compact group $G$ also let $\mu_l$ and $\mu_r$ be arbitrary $G$-invariant measures on $G/_lH$ and  $G/_rH$ respectively.
The involution-type map
$$^{*_{r,l}}:L^1(G/_rH,\mu_r)\to L^1(G/_lH,\mu_l),$$
is an anti-homomorphism isometric.}
\end{corollary}

%----------------------------------------------------------------------------------------------------------------------------------------

\begin{remark}
If $H$ is a compact normal subgroup of a locally compact group $G$, left and right coset spaces are the same and so that the involution-type maps given in (\ref{inv1}) and (\ref{inv2}) coincide with the standard involution on $L^1(G/H,\mu)$.
\end{remark}

%----------------------------------------------------------------------------------------------------------------------------------------

{\bf ACKNOWLEDGEMENTS.}
The author would like to thank the referees for their valuable comments and remarks.
He also would like to gratefully acknowledge financial support from the Numerical Harmonic Analysis Group (NuHAG) at the Faculty of Mathematics, University of Vienna. Especially, he would like to express his gratitude to Prof. Hans G. Feichtinger (group leader of NuHAG) for stimulating discussions and pointing out various references.

\bibliographystyle{amsplain}

\end{document}